\documentclass[11pt]{amsart}
\usepackage{amsmath, amssymb, graphics, graphicx}
\usepackage[latin1]{inputenc}
\usepackage{xcolor}
\usepackage{latexsym}
\usepackage{graphics,epsfig}
\usepackage{amsfonts}
\usepackage[english]{babel}
\usepackage{verbatim}
\newtheorem{theorem}{Theorem}
\newtheorem{lemma}[theorem]{Lemma}
\newtheorem{proposition}[theorem]{Proposition}
\newtheorem{corollary}[theorem]{Corollary}

\theoremstyle{definition}
\newtheorem{definition}{Definition}
\theoremstyle{remark}
\newtheorem{remark}{Remark}
\newcommand{\CFG}{\text{CFG}}
\newcommand{\SG}{\mathcal{SG}}
\begin{document}
\title[$G$-strongly positive scripts and critical configurations]
{$G$-strongly positive scripts and critical configurations of chip firing games on digraphs}
\author{TRAN Thi Thu Huong}
\email{ttthuong@math.ac.vn}
\address{Institute of Mathematics, Vietnam Academy of Science and Technology, 18 Hoang Quoc Viet, Cau Giay, Hanoi, Vietnam}
\maketitle
\begin{abstract}
We show a collection of scripts, called $G$-strongly positive scripts, which is used to recognize critical configurations of a $\CFG$ (chip firing game) on a multi-digraph with a global sink. To decrease the time of the process of recognition caused by the stabilization we present an algorithm to find the minimum $G$-strongly positive script. From that we prove the  non-stability of configurations obtained from a critical configuration by firing inversely any non-empty multi-subset of vertices. This result is a generalization of a very recent one by Aval \emph{et.al} which is applied for $\CFG$ on undirected graphs. Last, we give a combinatorial proof for the duality between critical and super-stable configurations. 
\end{abstract}
\section{Introduction}
$\CFG$ (chip firing game) is a game on a (di)graph $G$ introduced by Bjorner, Lov\'asz and Shor \cite{BLS91}: each vertex contains some chips, and a move consists of selecting a vertex with at least chips as many as its (out-going) degree and firing it by sending one chip along each  out-going edge from it.  The game converges if there is no such vertex. They consider the convergence of the game and prove that the convergent configuration is independent on the moves.

The version of $\CFG$ on undirected graphs also  appears earlier under the name Abelian Sandpile Model in \cite{Dha90} to study the criticality of some self-organized systems in physics. Herein he considers a class of configurations of $\CFG$, called critical configurations, which are recurrent by the addition of chips and applying several steps of moves. He proves that only critical configurations have a non-zero probability of occurrence by the Markovian evolution. Furthermore, this non-zero is the same for all critical configurations. Since then, critical configurations have been shown of plenty interesting properties relating to a wide variety of known objects such as spanning trees, sandpile group, $G$-parking functions, Tutte polynomial \cite{CP05,Mer97,RC00}.  

$\CFG$ on directed graph is introduced systematically in \cite{BL92}. Its critical configurations are investigated with nice results concerning to rooted spanning trees, rotor router in \cite{HLMPPW08}. Much of theory developed for undirected graphs works for directed graphs but not always. For recognizing critical configurations on undirected graphs, Dhar introduces burning algorithm \cite{Dha90}. Then Speer  \cite{Spe93} developed it into script algorithm for the recognition problem on digraphs with a global sink and another strongly connected component. He also shows the minimum testing script and proves that all its parts are equal to $1$ for digraphs with no selfish site with the noticing that undirected graphs belong this digraphs. 

In this manuscript, we give some characterizations of critical configurations of $\CFG$ whose support graph is directed with a global sink. Section 2 presents some basic definitions and properties of critical configurations of $\CFG$. In Section 3, we show a collection of scripts, called $G$-strongly positive scripts, which are used to recognize critical configurations. To find the minimum $G$-strongly positive script $\sigma^M$, for reducing the stabilization process in the recognition, we introduce the strongly script algorithm which consists of running the script algorithm several times (with a determined order) on strongly connected components of $G$. Using this minimum script, we prove a characterization of critical configurations via the non-stability of configurations obtained from them by inversely firing any multi-subsets of vertices of $G$ as well as any multi-subsets of $\sigma^M$. This result generalizes a result in \cite{ADDL13} which is applied for undirected graphs. At the end of the section, we give the affirmative answer for a question raised in \cite{PP14} concerning the maximum of weight of critical configurations among stable configurations in its equivalent class.  Last, we present in Section 4 a combinatorial proof for the duality between critical and super-stable configurations based on the non-stability characterization of critical configurations given in Section 3. 
\section{Preliminaries}
Let $G=(V,E)$ be a directed multi-graph (digraphs) with $n+1$ vertices $V=\{1,2,\dots,n+1\}$. We assume that $G$ has a \emph{global sink} at vertex $n+1$, \emph{i.e.}, the out-going degree of $n+1$ is equal to $0$ and there exists a directed path from any vertex other than $n+1$ to $n+1$. The vertex $n+1$ is called the \emph{sink} of $G$. Denote $d^+_i$, $d^-_i$ and $e_{ij}$ the \emph{out-going degree, in-going degree} of $i$ and the \emph{number of edges from $i$ to $j$} respectively in $G$. The
\emph{Laplacian matrix} $\tilde{\Delta}$ of $G$ is an $(n+1)\times (n+1)$ matrix
defined as follows
\begin{equation*}
\tilde{\Delta}_{ij}=
\begin{cases}
d^+_i \quad &\text { if } i=j,\\
-e_{ij} \quad & \text{ if } i\ne j.
\end{cases}
\end{equation*}

The \emph{reduced Laplacian matrix} of $G$, denoted by $\Delta$, is the matrix obtained from $\tilde \Delta$ by deleting its row and column $n+1$. Let $\Delta_i$ be the $i$th row-vector of $\Delta$. The \emph{Picard group},
also called \emph{Sandpile group}, of $G$ is defined by $$\SG(G)=\mathbb Z^n/\langle\Delta_1,\dots,\Delta_n\rangle.$$ 
We recall some facts on Laplacian and its reduced matrices: 
\begin{itemize}
\item $\tilde \Delta$ is not invertible and $rank\ \tilde\Delta=n$.
\item The sum of row-vectors of $\tilde\Delta$ is vector zero.
\item $\Delta$ is invertible and all entries of $\Delta^{-1}$ are non-negative.
\item The sum of entries on each row of $\Delta$ is non-negative. 
\end{itemize}
\begin{definition}
Let $a,b\in \mathbb Z^n$. We said that $a$ is \emph{linear equivalent} to $b$, denoted by $a\sim b$, if they are in a same coset of $\SG(G)$. 
\end{definition}

Notice that since $\Delta$ is invertible, if $a\sim b$ then $b=a+\tau \Delta$ where $\tau=(b-a)\Delta^{-1}$ is determined unique. 

Now let $G$ be a digraph with a global sink. A $\CFG$ (\emph{chip firing game}) is defined on $G$, $G$ is call support of $CFG$. It includes the set of configurations and a firing rule transferring between configurations. \emph{A configuration} of a $\CFG$ on $G$ is a map  $$a: V\setminus n+1\to \mathbb Z$$ and $a(i)$ is the number of chips stored at vertex $i$. Sometimes we prefer to write a configuration on $G$ as an element of $\mathbb Z^n$ where part $i$ is equal to $a(i)$. Vertex $i$ of a configuration $a$ is called \emph{active} if $a(i)\geq d^+_i$. The firing rule is applied at an active vertex $i$ and it will pass $d^+_i$ chips stored at $i$ on its neighbors along its out-going edges (taking into account the multiple edges). As a sequence,  if we apply the firing rule at $i$, or say simple firing $i$, of $a$, chips on $a$ are redistributed to a new one $b$ such that
\begin{equation*}
b(j)=
\begin{cases}
a(j)-d^+_i &\text{ if } j=i,\\
a(j)+e_{ij} & \text { if  otherwise,}
\end{cases}
\end{equation*} 
or equivalently, 
$$b=a-\Delta_i,$$ where recall that $\Delta_i$ is the row $i$ of the reduced Laplacian matrix of $G$. 

A configuration $a$ is called \emph{non-negative} if $a(i)\geq 0$ for all $i=1,2,\dots, n$. Moreover, $a$ is called \emph{stable} if $a$ is non-negative and $a(i)$ does not exceed more than $d^+_i-1$ for $i=1,2,\dots,n$. The \emph{weight} of $a$, denoted by $w(a)$, is the sum of all parts of $a$.

From definition, sometimes chips at a vertex will be transferred to the sink and so the sink could contain some chips. However, in this manuscript when considering configurations we discard chips stored at sink. Furthermore, there is no any vertex of a stable configuration which is active. 
\begin{definition}
\begin{itemize}
\item[i)] A sequence of vertices $(s_1,\dots,s_k)$ of $V\setminus n+1$ is called a \emph{firing sequence} of $a$ if vertex $s_1$ of $a$ is active and after firing consecutively from $a$ vertices $s_1,s_2,\dots,s_i$ the vertex $s_{i+1}$ is active for each $i=2,\dots, k-1$. 
\item[ii)] A \emph{firing script} of a firing sequence $(s_1,s_2,\dots,s_k)$ of $\CFG$ on $G$ is a sequence of $n$ non-negative integers $\tau=(\tau_1,\tau_2,\dots,\tau_n)$, where $\tau_i$ is the number of occurrences vertex $i$ in the firing sequence $(s_1,s_2,\dots,s_k)$. 
\end{itemize}
\end{definition}

It is remarkable that if we apply the firing sequence $(s_1,s_2,\dots,s_k)$ on $a$ with the corresponding firing script $\tau$, we get $b=a-\tau\Delta$ and so $b$ is linear equivalent to $a$. The inverse does not hold in general. It is a fact that if $b\sim a$, then $b=a-\tau\Delta$ with $\tau=(a-b)\Delta^{-1}$ which is not necessary non-negative. Even $\tau$ is non-negative, it does not make sure that there exists a firing sequence from $a$ to $b$.

Notice that from a non-negative configuration $a$, we could apply several steps of firing rule on $a$ to get a stable configuration. The process of firing from $a$ to a stable configuration is called \emph{a stabilization} of $a$. 
\begin{lemma}[\cite{LP01}] Let $G$ be a digraph with a global sink and $a$ be a non-negative configuration. Then the system $\CFG(G)$ starting from $a$ converges to a unique stable configuration. Furthermore, if $s$ and $s^\prime$ are two different firing sequences in the stabilization of $a$, then $s$ and $s^\prime$ have the same firing script. 
\end{lemma}

Denote $a^o$ the unique stable configuration obtained from $a$ by firing vertices. By the lemma above although there may have many firing sequences in the stabilization, there is unique firing script. In this manuscript we also consider sequences whose part $i$ is the number of occurrences of vertex $i$ and so their parts are non-negative. These sequences are very likely to the firing scripts and they play an important role in our next  investigation on critical configurations of $\CFG$. We refer a sequence of non-negative integers for such purpose as a \emph{script} and as an \emph{$n$-script} if its length is $n$. It is readily that a firing script of a configuration is a script but as noticed above a script is not necessary a firing sequence.  

Before presenting the definition of critical configurations of $\CFG$ we need some notations and operations on $\mathbb Z^n$. Denote $\succeq$ the containment order in $\mathbb Z^n$, that means $a\succeq
b$ if $a_i\geq b_i$ for all $i=1,2,\dots,n$. The element $(0,0,\dots,0)\in \mathbb Z^n$ is denoted by $\vec 0$. The \emph{support} of an element $a\in \mathbb Z^n$, denoted by $supp(a)$, is defined by $$supp(a):=\{i:a_i\ne 0\}.$$  The addition of two sequences and the scalar product with a number $\delta\in \mathbb Z$ are implemented in $\mathbb Z^n$ as follows:
$$a+b=(a_1+b_1,\dots,a_n+b_n),$$ and $$\delta a=(\delta a_1,\delta a_2, \dots,\delta a_n).$$
For convenience of writing expressions, we also write $$a+\delta=(a_1+\delta, a_2+\delta, \dots,a_n+\delta).$$
Let $$a_{\max}=(d^+_1-1,d^+_2-1,\dots,d^+_n-1).$$ Then $a_{\max}$ is the maximum stable configuration of $\CFG(G)$ with the containment order. 
\begin{definition}
A configuration $a$ of $\CFG(G)$ is called \emph{critical} if it is stable and there exists a non-negative configuration $c$ such that $a=(a_{\max}+c)^o$.
\end{definition} 
  
We recall some facts on critical configurations:
\begin{lemma}[\cite{HLMPPW08}] \label{L:HLM1} Each equivalent class of $\mathcal{SG}(G)$ contains exactly one critical configuration. 
\end{lemma}
\begin{lemma}[\cite{HLMPPW08}] \label{L:HLM2} Let $a$ be a configuration on $\CFG(G)$. The following statements are true:
\begin{itemize}
\item[i)] Let $b$ be a stable configuration of $\CFG$ and $b\succeq a$. If  $a$ is critical, then $b$ is critical too. 
\item[ii)] $a$ is critical if and only if for any configuration $b$ there exists a non-negative configuration $c$ such that $a=(b+c)^o$.
\item[iii)] The configuration $\big( a+(2a_{\max}+2)-(2a_{\max}+2)^o\big)^o$ is critical. Moreover, it is the unique critical configuration in the equivalent class of $a$.   
\end{itemize}
\end{lemma}

Notice that there may have some equivalent definitions of critical configurations on digraphs. Holroyd \emph{et.al.} use the sufficient statement of Lemma \ref{L:HLM2}(ii) as a definition for the criticality \cite{HLMPPW08}. Lemma \ref{L:HLM2}(iii) could be used to find the critical configuration in the same equivalent class of $a$. However, this process could take time since the stabilization is exhausted. 
\section{$G$-strongly positive scripts and critical configurations}
In this section we introduce a collection of scripts which is used to recognize critical configurations for digraphs with a global sink. The class of these scripts is called $G$-strongly positive script $\sigma^M$. For reducing the process of stabilization in the recognition, we also present an algorithm to find the minimum $G$-strongly positive script. This algorithm is a concatenation of script algorithms \cite{Spe93} on strongly connected components of $G$. Using this script, we give a non-stability characterization of critical configurations when firing inversely multi-subsets of vertices of $G$ as well as of $\sigma^M$ (Theorem \ref{T:script}). This property will be used for the proof of the duality between critical and super-stable configurations combinatorially next section. We first present a property of firing scripts.
\begin{lemma}\label{L:prec}
Let $\sigma=(\sigma_1,\dots,\sigma_n)$ be an $n$-script. Let $a$ be a stable
configuration of $\CFG(G)$. Let $\tau$ be the firing script in the
stabilization process of $a+\sigma\Delta$. Then $\tau\preceq \sigma$.
\end{lemma}
\begin{proof}
Let $b=a+\sigma\Delta$. Since $a$ is stable, we have 
$$0\leq a_i=b_i-\sigma_i\Delta_{ii}+\sum_{j=1}^n\sigma_j\Delta_{ji}< \Delta_{ii}, \text{ for all } i=1,2\dots,n.$$ Let
$s=(s_1,s_2,\dots,s_k)$ be a firing sequence in the stabilization process of $b$. Notice that $s_i$ could repeat. Denote $s^{\leq t}=(s_1, s_2,\dots,s_t)$ the firing sequence of $s$ till the time $t$. Let $\tau^{\leq t}$ be the $n$-script associated to firing sequence $s^{\leq t}$, \emph{i.e.} $\tau^{\leq t}_j$ equal to the occurrences of vertex $j$ in the sequence $s^{\leq t}$ for $j=1,2,\dots,n$. Assume to the contrary that $\tau \npreceq \sigma$. Let $t_0$ be the first time in the processing of firing sequence $s$ such that the number of occurrences of vertex $s_{t_0}$ in the firing sequence $s^{\leq t_0}$ greater than $\sigma_{s_{t_0}}$. Let $i=s_{t_0}$ then we have 
\begin{equation*}
\begin{cases}
\tau^{\leq t_0}_j=\sigma_j+1 &\text{ if } j=i,\\
\tau^{\leq t_0}_j\leq \sigma_j &\text{ if } j\in \{1,2,\dots,n\} \text { and }j\ne i.
\end{cases}
\end{equation*} Therefore, after firing $s_{t_0}$ in the firing sequence $s$, we consider the chips at vertex $i$. Since $b$ is stable and $\Delta_{ij}<0$ for $i\ne j$, we have the following evaluation: 
$$b_i-\tau^{\leq t_0}_t\Delta_{ii}-\sum_{j\ne i} \tau^{\leq t_0}_j
\Delta_{ji}\leq b_i-(\sigma_i+1)\Delta_{ii}-\sum_{j\ne i}
\sigma_j\Delta_{ji}<0.$$ This is a contradiction to
the condition that $s$ is a firing sequence.
\end{proof}

Before presenting a criterion for the criticality of a configuration on digraph in Theorem \ref{P:cri} which is very similar to Dhar's algorithm, we need some definitions. 
\begin{definition}Let $G$ be a digraph with a global sink and reduced Laplacian matrix $\Delta$. Let $\sigma$ be an $n$-script. Then  
\begin{itemize}  
\item[i)] $\sigma$ is called a \emph{$G$-positive script} if $\sigma\Delta \succ \vec 0$
\item[ii)] $\sigma$ is called a \emph{$G$-strongly positive script} if $\sigma$ is $G$-positive script and the support of $\sigma\Delta$ intersects to each strongly connected component except for the sink component of $G$ non-empty.   
\end{itemize}
\end{definition}
\begin{remark}\begin{itemize}
\item[-] In case $G$ contains one strongly connected component different from the sink, then the properties of $G$-positive and $G$-strongly positive coincide. 
\item[-] A $G$-strongly positive script is a $G$-positive script but the inverse is not true. For instance, let $G$ be given as Figure \ref{F:critdigraph}. Then the reduced Laplacian matrix of $G$ is given by 
\begin{equation*}
\Delta= \left(
\begin{matrix}
7&-6&0\\
-1&4&-3\\
0&0&2
\end{matrix}
\right).
\end{equation*}
Hence, $G$ has three strongly connected components $\{v_1,v_2\}, \{v_3\},\{s\}$. We have $(1,2,3)\Delta=(5,2,0)$ and $(0,0,1)\Delta=(0,0,2)$. Hence, $(1,2,3)$ and $(0,0,1)$ are $G$-positive scripts. However, both $(1,2,3)$ and $(0,0,1)$ are not  $G$-strongly positive since the support of $(1,2,3)\Delta$ (resp. $(0,0,1)\Delta$) is $\{v_1,v_2\}$ (resp. $\{v_3\}$) which intersects to the strongly connected component $\{v_3\}$ (resp. $\{v_1,v_2\}$) empty.  
\item[-] The positive scripts and strongly positive scripts of a non-trivial graph with a global sink exist. Thus, let $\sigma=a\Delta^{-1}$. Therefore, we could choose many $a$ such that the parts of $\sigma$ are integer. Notice that $\Delta^{-1}$ is non-negative matrix. Hence, if $a$ is non-negative, then $\sigma$ is $G$-positive script. If all parts of $a$ are positive, then $\sigma$ is $G$-strongly positive.    
\end{itemize}
\end{remark}
\begin{figure}[h]
\centering
\includegraphics[scale=0.8]{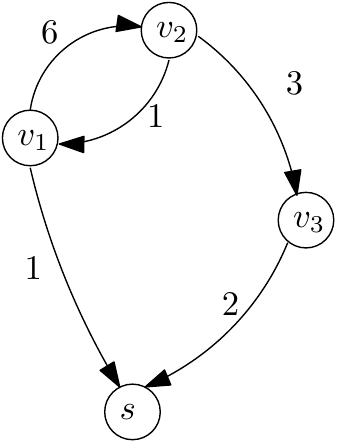}
\caption{A graph with global sink $s$}\label{F:critdigraph}
\end{figure}
The following lemma gives us an evaluation on the weight of a configuration when firing inversely it by a positive script. 
\begin{lemma}\label{L:gre} Let $\sigma$ be a $G$-positive script and let $a,b$ be non-negative configuration such that $b=(a+\sigma\Delta)^o$. Then $w(b)\geq w(a)$.
\end{lemma}
\begin{proof}
Let $\tau$ be the firing script in the stabilization process of
$(a+\sigma\Delta)$. Recall that in the stabilization of a configuration, chips which are sent to the sink never come back. Hence, 
$$w(b)=w(a+\sigma\Delta)-\sum_{i: (i,n+1)\in E} \tau_i e_{i(n+1)}.$$ Furthermore, since $\sum_{i=1}^{n+1}\tilde{\Delta}_i=0,$ we have   
$$w(a+\sigma\Delta)=w(a)+\sum_{i:(i,n+1)\in E}\sigma_i e_{i(n+1)}.$$
By Lemma \ref{L:prec}, we have $\sigma_{i}\geq \tau_{i}.$ Hence, $w(b)\geq
w(a)$.
\end{proof}

The statement of the following result appeared in a different form in \cite{PPW11}. However, since we did not find any proofs for it in the literature, we restate it with the full proof as below.    
\begin{theorem}\label{P:cri}
Let $a$ be a stable configuration and $\sigma$ be a $G$-strongly positive script. Then $a$ is critical if and only if $(a+\sigma\Delta)^o=a$.
Moreover, if $a$ is critical then the firing script in the stabilization
process of $(a+\sigma\Delta)$ is $\sigma$.
\end{theorem}
\begin{proof}
We first prove that $(a_{\max}+\sigma\Delta)^o=a_{\max}$. By Lemma
\ref{L:gre}, we have $w(a_{\max}+\sigma\Delta)^o\geq w(a_{\max})$. Since
$(a_{\max}+\sigma\Delta)^o$ is stable, we must have
$(a_{\max}+\sigma\Delta)^o\preceq a_{\max}$. Therefore,
$(a_{\max}+\sigma\Delta)^o= a_{\max}$. 
Now let $a$ be critical. By definition, there exists $c\succeq 0$ such that $a=(a_{\max}+c)^o$. We have
\begin{align*}
(a+\sigma\Delta)^o &=((a_{\max}+c)^o+\sigma\Delta)^o\\
&=(a_{\max}+c+\sigma\Delta)^o \quad \quad \text{(since } c\succeq \vec 0)\\
&=((a_{\max}+\sigma\Delta)^o+c)^o \quad\quad\text{(since } \sigma\Delta\succeq \vec 0)\\
&=(a_{\max}+c)^o\\
&=a.
\end{align*}

Conversely, assume that $(a+\sigma\Delta)^o=a$. Since $\sigma$ is $G$-strongly positive, the support of $\sigma\Delta$ on each strongly connected component are non-zero. There exists a large enough $m$ such that after
applying several firings on $(a+m\sigma\Delta)$ we could obtain a
configuration greater (with respect to the containment order) than $a_{\max}$.
This means there exists $c\succ 0$ and a firing sequence of firing script $k$ such that starting from
$a+m\sigma\Delta$, after firing by the firing script $k$ we get 
$$a_{\max}+c=(a+m\sigma\Delta)-k\Delta.$$ 
Since $(a+\sigma\Delta)^o=a$ and $\sigma\Delta \succeq \vec 0$, we have 
$$(a+2\sigma\Delta)^o=(a+\sigma\Delta+\sigma\Delta)^o=((a+\sigma\Delta)^o+\sigma\Delta)^o=(a+\sigma\Delta)^o=a.$$
Therefore, 
$$a=(a+\sigma\Delta)^o=\big(a+m\sigma\Delta\big)^o =\big((a+m\sigma\Delta)-k\Delta\big)^o=(a_{\max}+c)^o.$$
So $a$ is critical.

Last, let $\tau$ be the firing script in the stabilization process of $a+\sigma\Delta$. Since $a$ is critical, we have 
$(a+\sigma\Delta)^o=a$. Hence, $a+\sigma\Delta-\tau\Delta=a$ and so $(\sigma-\tau)\Delta=0$. Since $\Delta$ is invertible, we have $\sigma=\tau$.
\end{proof}

It remarks that we cannot replace the condition $G$-strongly positive by the $G$-positive of $\sigma$  in Theorem \ref{P:cri}. For instance, consider the graph given in Figure \ref{F:critdigraph}. Let $\sigma=(0,0,1)$ and $\sigma^\prime=(1,2,4)$. We have $\sigma \Delta=(0,0,2)$ and $\sigma^\prime \Delta=(5,2,2)$. Hence,  $\sigma$ is $G$-positive but not $G$-strongly positive and $\sigma^\prime$ is $G$-strongly positive. Let $a=(1,1,1)$ be a configuration of $\CFG$ on that graph. We have $(a+\sigma^\prime\Delta)^o=(6,3,3)$ and by Theorem \ref{P:cri}, $a$ is not critical. However, if we replace $\sigma^\prime$ by $\sigma$, then $(a+\sigma\Delta)^o=(1,1,1)=a$ which is recurrent. 

Recall that Speer \cite{Spe93} gave an algorithm, called \emph{script algorithm}, to find the minimum $G$-positive script in case $G$ has only one strongly connected component different from the sink. From that an algorithm like "burning algorithm" \cite{Dha90} for testing the criticality of a configuration was presented. Moreover, for general $G$ (with the global sink), he showed that each critical configuration of $G$ is a concatenation of critical configurations of some certain digraphs with a global sink and one strongly connected component different from the sink, which are induced by the strongly connected components of $G$. However, he did not show explicit the minimum testing script. In fact, this testing script is not a concatenation of the minimum scripts on each its induced component. The reason for this un-wanted fact is that the firing caused by the minimum script on an induced component affects to the firing on another induced component. To improve this we use the (modifying) script algorithm for each induced component with a given order (will be defined later) such that the inverse firing of the latter does not affect to the positiveness of the former ones. We first modify the script algorithm for finding the minimum script $\sigma$ such that $\sigma\Delta\succ a$ with a given non-negative configuration $a$ on a digraph having one strongly connected component different from the sink. The algorithm, called \emph{$a$-script algorithm}, is as follows: We construct a sequence of increasing scripts $\{\sigma^{k}_a\}_k$ by recurrence and will prove it terminates at step $M$ to obtain the  script $\sigma^M_a$ (which will be the minimum script). We start  from $\sigma^1_a=(1,0,\dots,0)$. Assume that we had $\sigma^{k-1}_a$. At step $k$, if there exists a vertex $i_k$ such that $(\sigma^{k-1}_a\Delta)_{i_k} < a_{i_k}$, then we increase part $i_k$ of $\sigma^{k-1}_a$ by one. Hence, $\sigma^{k}_a=\sigma^{k-1}_a+e_{i_k}$, where $e_{i_k}$ denotes the unit vector at the $i_k$th coordinate. The following lemma shows the existence and uniqueness of $\sigma^M_a$. 
\begin{lemma}\label{L:scr} Let $G$ be a digraph with a global sink. Assume that $G$ contains only one strongly connected component except for the sink. Let $a$ be non-negative configuration of $\CFG$ on $G$. The following statements hold: 
\begin{itemize}
\item[i)] The $a$-script algorithm above terminates and results the unique script $\sigma^M_a$;
\item[ii)] $\sigma^M_a\succeq (1,1,\dots,1)$ and $(\sigma^M_a\Delta)_i \leq a_i+d^+_i$ for all $i=1,2,\dots,n$. Moreover, the equality of the latter happens if and only if $n=1$; 
\item[iii)] If $\sigma$ is a script such that $\sigma\Delta\succ a$, then $\sigma \succeq \sigma^M_a$. 
\end{itemize}
\end{lemma}
\begin{proof} The proof of this lemma is very similar to the one presented in \cite{Spe93}(Lemma 7) which is equivalent to the case $a=\vec 0$. We represent shortly the proof as follows: 
\begin{itemize}
\item[i)] We use the \emph{diamond property} of the sequence $\{\sigma^i_a: i=1,2,\dots,M\}$ for $a\succeq \vec 0$: The increasing part $i$ at each step by the recurrence in the $a$-script algorithm does not decreasing the ability of the "having to increase" the part $j$ at next steps. This fact will imply the termination and uniqueness of $\sigma^M_a$. 
\item[ii)] The increasing a part $i$ of a script in the construction absorbs chips of vertices going-out from $i$: The strongly connectivity of $[V\setminus s]$ and the non-negativity of all parts of $a$ lead to $\sigma^M_a\succeq (1,1,\dots,1)$. The rest of $(ii)$ is clear because of the construction of $\sigma^M_a$. 
\item[iii)] By induction on $i$ that $\sigma\succeq \sigma^i_a$ for all $i=1,2,\dots,M$. 
\end{itemize}  
\end{proof}

By Theorem \ref{P:cri}, we will find the minimum $G$-strongly positive script for general digraphs with a global sink. Before presenting the algorithm, we need some definitions and notations. 

Assume that $$V=V_1\sqcup V_2\cdots\sqcup V_k\sqcup \{n+1\},$$ where $V_i$ are strongly connected components of $G$. We say that component $V_i$ is \emph{at level $0$} if there is no edge of $G$ to $V_i$ from out side  $V_i$. By recurrence, component $V_i$ is \emph{at level $\ell$} if there is at least one edge of $G$ from a component at level $\ell-1$ to $V_i$ and there is no edge of $G$ to $V_i$ from components out side components at level $0,1,2,\dots,\ell-1$. 
Let $G[V_i]$ be the extended graph defined as follows: The vertices are $V_i\cup s_i$ where $s_i$ is a new vertex which will be the global sink of $G[V_i]$; The edges are edges of $G$ with endpoints in $V_i$ and all edges $(v,s_i)$ (including the multiple edges) for each $(v,u)$ is an edge from $v$ to a vertex $u$ out side $V_i$. It is easy to see that $G[V_i]$ has a global sink $s_i$. To  find the minimum $G$-strongly positive script for general $G$, we implement the algorithm, called \emph{strongly script algorithm}, as follows:
\begin{itemize}
\item Step 1: Find all the minimum $G[V_i]$-strongly positive script $\sigma^{M_i}_{\vec 0}(G[V_i])$ of each strongly connected components $V_i$ at level $0$ of $G$. 
\item Step 2: Update the last configuration, called $a$, by firing inversely from configuration $\vec 0$ all vertices of $V_i$ by the script $\sigma^{M_i}_{\vec 0}(G[V_i])$ for all $V_i$ at level 0. 
\item Step 3: Running the $a|V_i$-script algorithms on the graph $G[V_i]$ to find $\sigma^{M_i}_{a|V_i}(G[V_i])$ for all $V_i$ at level $1$ with $a|V_i$ denoted the configuration $-a$ restricted on component $V_i$.
\item Step 4: Update the last configuration by firing inversely from $a$ all vertices of $V_i$ by the script $\sigma^{M_i}_{a|V_i}$ for all $V_i$ at level 1. 
\item Continuing the step 3 and 4 sequentially for components at level 2,3,... until all strongly connected components are well fired inversely. The script obtained by assembling all the scripts of all connected components found above. This script is denoted by $\sigma^M$.  
\end{itemize}
 
The following facts on the minimum of the script $\sigma^M$ are implied easily from Lemma \ref{L:scr}:
\begin{proposition} The script $\sigma^M$ satisfies the followings: 
\begin{itemize}
\item[i)] $\sigma^M$ is the minimum $G$-strongly positive script of $G$; 
\item[ii)] $\sigma^M\succeq (1,1,\dots,1)$ and $(\sigma^M\Delta)_i\leq d^+_i$ for all $i=1,2,\dots,n$. Moreover, $(\sigma^M\Delta)_i=d^+_i$ if and only if $i$ is a source of $G$, that means $d^-_i=0$. 
\end{itemize}
\end{proposition}
We have shown a way to determine the minimum $G$-strongly positive script by inversely firing strongly connected components in an order from components of low levels to high levels. Components are fired (inversely) first like the argents causing the consecutive firings of components of higher levels. Comparing to the Dhar's algorithm when the minimum script is $1$ at every parts, firing the sink first in Dhar's algorithm is equivalent to firing inversely all vertices and each vertices once. 

\begin{figure}[h] \label{F:enc}
\begin{center}
\includegraphics{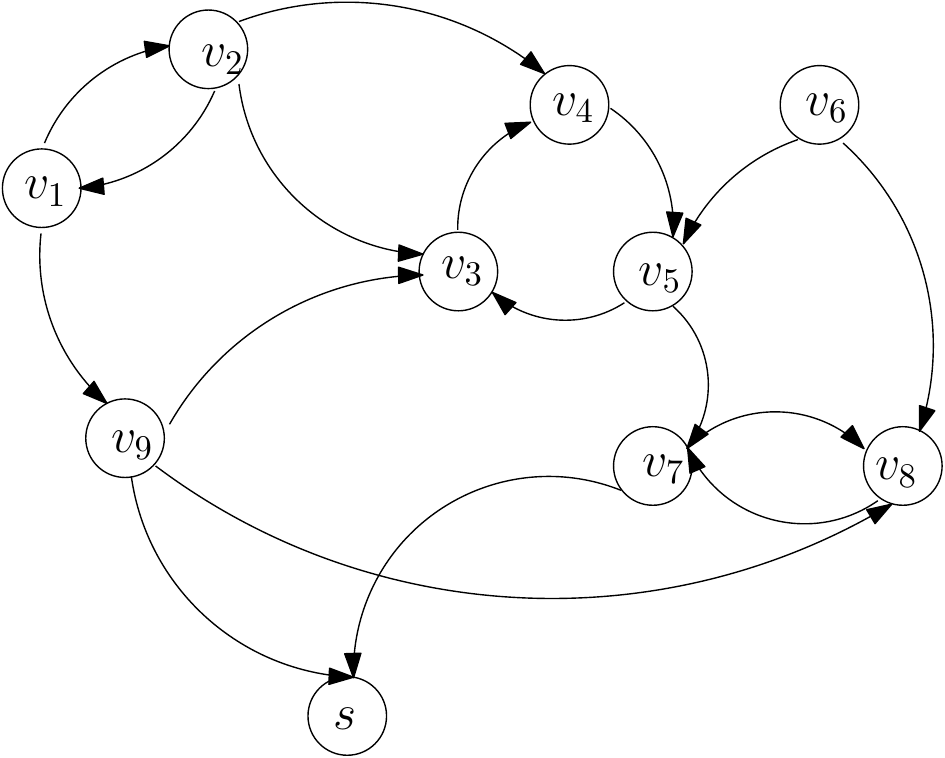}
\caption{Graph with many strongly connected components}
\end{center}
\end{figure} 
Let us consider an example for illustrating the algorithm. Figure \ref{F:enc} shows a digraph $G=(V,E)$ of $9$ normal vertices and a global sink $s$. $G$ has the decomposition into strongly connected components: $$V=\{v_1,v_2\}\sqcup \{v_3,v_4,v_5\}\sqcup \{v_6\}\sqcup \{v_7,v_8\}\sqcup \{v_9\}\sqcup\{s\}.$$
Components $\{v_1,v_2\}$ and $\{v_6\}$ are at level $0$; Component $\{v_9\}$ is at level $1$; Component $\{v_3,v_4,v_5\}$ is at level $2$ and component $\{v_7,v_8\}$ is at level $3$. By the strongly script algorithm above, we find strongly positive scripts for components at level $0$: the result script for $G[v_1,v_2]$ is $(1,1)$; for $G[v_6]$ is $(1)$; The update configuration for $G[v_9]$ by inversely firing in result scripts of components at level $0$ is $(-1)$ and so the $(1)$-script algorithm for $G[v_9]$ returns the script $(1)$. The update configuration on $G[v_3,v_4,v_5]$ by inversely firing in result scripts of components at level $0$ and $1$ is $(-2,-1,-1)$ and so the corresponding script is $(7,8,5)$. The update configuration on $G[v_7,v_8]$ is $(-5,-2)$ and the corresponding script is $(8,10)$ by $(5,2)$-script algorithm for $G[v_7,v_8]$. Hence, the minimum $G$-strongly positive script for digraph in Figure \ref{F:enc} is $(1,1,7,8,5,1,8,10,1)$. 

Next, we use the minimum $G$-strongly positive script to present a result which generalizes the result \cite{ADDL13} for undirected graphs. Moreover, instead of firing a single subset of vertices it fires a multi-subset of its vertices as well as a subset of a multi-set $\sigma^M$. This helps us to prove combinatorially the duality between critical and super-stable configuration which the firing a multi-subset is not allowed.
\begin{theorem}\label{T:script}
Let $\sigma^M$ be the min $G$-strongly positive script of $\CFG(G)$. The following statements are equivalent: 
\begin{itemize}
\item[i)] $a$ is critical;
\item[ii)] $a+\tau\Delta$ is not stable for all $\tau\succ \vec 0$;
\item[iii)] $a+\tau\Delta$ is not stable for all $\vec 0\prec \tau\preceq \sigma^M.$
\end{itemize}
\end{theorem}
\begin{proof}
We will show that $(i)\Rightarrow (ii)\Rightarrow (iii) \Rightarrow (i)$.
First, assume that $a$ is critical and there exists a script $\tau\succ \vec 0$ such
that $a+\tau\Delta$ is stable. Let $b=a+\tau\Delta$ and $m$ be a large enough integer such that $m> \max\{\tau_i: i=1,2,\dots,n\}$. Put $\sigma=m\sigma^M$. Since $\sigma^M\succeq (1,1,\dots,1)$, we have $\sigma\succ \tau$. On the other hand, since $\sigma^M$ is $G$-strongly positive, $\sigma$ is also $G$-strongly positive. By Theorem \ref{P:cri}, we have $(a+\sigma\Delta)^o=a$.
Therefore, 
$$a+\sigma\Delta=a+\tau\Delta + (\sigma-\tau)\Delta=b +(\sigma-\tau)\Delta.$$ By Lemma
\ref{L:prec} and due to the stability of $b$, the firing script in the stabilization process of $b+(\sigma-\tau)\Delta$ does not exceed $\sigma-\tau$ and so does the firing script in the stabilization process of $a+\sigma\Delta$. This conflicts the critical condition of $a$ by Theorem \ref{P:cri}. 

$(ii)\Rightarrow (iii)$: Straightforward.

$(iii)\Rightarrow (i)$: Suppose that $a$ is not critical. Let $b=(a+\sigma^M\Delta)^o$. We assume that $b\ne a$. Let $\tau$ be the script of a firing sequence in the stabilization process of $(a+\sigma^M\Delta)$. By Lemma \ref{L:prec}, we have $\tau\prec \sigma^M$ (since $a\ne b$) and $b=a+\sigma^M\Delta-\tau\Delta$. Let $\nu=\sigma^M-\tau$. We have $0\prec\nu\preceq \sigma^M$ and $a+\nu\Delta$ is stable which conflicts the hypothesis.
\end{proof}

Perrot and Trung \cite{PP14} raised a question that "do critical configurations have the maximum weights among stable configurations in their equivalent class?". They gave an affirmative answer for $\CFG$ on Eulerian graphs and also claimed that critical configurations are not the only ones in their equivalent classes having the maximum weight by an encounter example. We end this section by proving the claim for $\CFG$ on graphs with a global sink.
\begin{theorem} \label{T:max} Let $a$ be a critical configuration and $b$
be a stable configuration in the same equivalent class with $a$. Then $w(a)\geq w(b)$.
\end{theorem}
\begin{proof} Let $\sigma$ be a $G$-strongly positive script. We consider the sequence of stable configurations as follows:
$(b+\sigma\Delta)^o,(b+2\sigma\Delta)^o,(b+3\sigma\Delta)^o,\dots$. Since the number of stable configurations of $\CFG(G)$ is finite, by the pigeonhole principle there exist two indices $k$ and $l$ such that $k< \ell$ and 
$(b+k\sigma\Delta)^o=(b+\ell \sigma\Delta)^o$. Denote $b^{(i)}=(b+i\sigma\Delta)^o$. For each $i$,
we have
\begin{align*}
b^{(i+1)}&=(b+(i+1)\sigma\Delta)^o\\
&=(b+i\sigma\Delta+\sigma\Delta)^o\\
&=\big((b+i\sigma\Delta)^o+\sigma\Delta\big)^o \quad \text{ (since }
\sigma\Delta\succeq 0)\\
&=(b^{(i)}+\sigma\Delta)^o.
\end{align*}
So that if we let $\tau^{(i)}$ be the scripts in the stabilization of $b^{(i)}+\sigma\Delta$. However, by Lemma \ref{L:prec} we have $\tau^{(i)}\preceq\sigma$ and $$b^{(i+1)}=b^{(i)}+\sigma\Delta-\tau^{(i)}\Delta.$$
Generally, 
$$b^{(k)}=b^{(\ell)}=b^{(k)}+(\ell-k)\sigma\Delta-(\tau^{(k)}+\tau^{(k+1)}+\cdots+\tau^{(\ell-1)})\Delta.$$
Hence,
$$(\ell-k)\sigma\Delta-(\tau^{(k)}+\tau^{(k+1)}+\cdots+\tau^{(\ell-1)})\Delta=\vec 0.$$ Since
$\Delta$ is invertible, we have
$(\ell-k)\sigma=(\tau^{(k)}+\tau^{(k+1)}+\cdots+s^{(\ell-1)})$. However, $\tau^{(i)}\preceq \sigma$ for $i=k,k+1,\dots,\ell$ which implies
that $\tau^{(k)}=\tau^{(k+1)}=\cdots=\tau^{(\ell-1)}$ and so
$b^{(k)}=b^{(k+1)}=\cdots=b^{(\ell)}$. Particularly, the firing script in the stabilization process of $b^{(k)}+\sigma\Delta$ is exactly $\sigma$ which confirms that
$b^{(k)}$ is critical. Moreover, since $b^{(k)}\sim a$, $b^{(k)}=a$. By Lemma
\ref{L:gre}, $w(a)=w(b^{(k)})\geq w(b)$ which completes the proof.
\end{proof}
\section{The duality between critical and super-stable configurations}
As mentioned, this section presents a combinatorial proof for the duality of critical configurations with 
super-stable configurations by using Theorem \ref{T:script}. As a sequence, we give a deterministic way for the super-stabilization process of a given configuration by using the minimum script. 

We first recall that a configuration $a$ of $\CFG(G)$ is non-negative if $a\succeq 0$.
\begin{definition} Let $a$ be a non-negative configuration on $\CFG(G)$.
We say that $a$ is super-stable if for all script $\sigma\in (\mathbb{Z}_+)^n$
and $\sigma\succ 0$ we have $a-\sigma\Delta$ is not non-negative.
\end{definition}
It is readily that if $a$ is super-stable and $b\preceq a$, then $b$ is
super-stable too. We have the duality between critical and super-stable
configurations as follows:
\begin{theorem}[Duality Theorem]\label{T:dua}
Let $a$ be a stable configuration. Then $a$ is critical if and only if
$a_{\max}-a$ is super-stable.
\end{theorem}
\begin{proof}
Assume that $a$ is critical and $a_{\max}-a$ is not super-stable. By the
definition, there exists a script $\sigma\succ 0$ such that
$a_{\max}-a-\sigma\Delta$ is non-negative. Let
$b=(a_{\max}-a-\sigma\Delta)^o$ and $\tau$ be the firing script in the stabilization process of $a_{\max}-a-\sigma\Delta$.
We have
\begin{align*}
&b=a_{\max}-a-\sigma\Delta-\tau\Delta\\
\Leftrightarrow \quad & a+(\sigma+\tau)\Delta=a_{\max}-b.
\end{align*}
Since $b$ is stable, $a_{\max}-b$ is stable. Therefore, $a+(\sigma+\tau)\Delta$
is stable which conflicts Theorem \ref{T:script}(ii).

Conversely, assume that $a$ is super-stable. We prove that $a_{\max}-a$ is
critical. On the contrary, by Theorem \ref{T:script}(ii) there exists an $n$-script
$\sigma$ such that $a_{\max}-a+\sigma\Delta$ is stable. Put
$b=a_{\max}-a+\sigma\Delta$. We have $a-\sigma\Delta=a_{\max}-b$ which is a
non-negative configuration (since $b$ is stable). By the definition, $a$ is not
super-stable. This completes the proof.
\end{proof}

Notice that the duality between critical and super-stable configurations on undirected as well as Eulerian graphs is mentioned in \cite{GK14,HLMPPW08}. On digraphs with a global sink, it is proved in \cite{PPW11}. However, the proofs are based on the third intermediate object \emph{sandpile group}. They all show that the number of critical configurations and the number of super-stable configurations are all equal to the order of sandpile group by using toppling ideal. The proof presented above seems more natural and quite direct from the definition of super-stable configurations. 
 
The following corollary allows us to reduce the checking definition of a super-stable configuration. 
\begin{corollary}\label{C:dua}
Let $a$ be a stable configuration and $\sigma^M$ be the minimum $G$-strongly positive script. The following statements are equivalent:
\begin{itemize}
\item[i)] $a$ is super-stable;
\item[ii)] $a-\sigma\Delta$ is not non-negative for all $0\prec\sigma \leq \sigma^M$;
\item[iii)] $a-\sigma\Delta$ is not stable for all $0\prec \sigma\preceq \sigma^M$.
\end{itemize}
\end{corollary}
\begin{proof} 
(i) $\Rightarrow$ (ii): By the definition of super-stable configuration;

(ii) $\Rightarrow$ (iii): Straightforward;

(iii) $\Rightarrow$ (i): Straightforward from Theorems \ref{T:script}(iii) and \ref{T:dua}.
\end{proof}

The next corollary shows us the minimum of weight of super-stable
configurations in its equivalent class.
\begin{corollary} Each equivalent class of $\mathcal{SP}(G)$ contains exactly one super-stable configuration. Furthermore, if $a$ is a super-stable configuration, then $w(a)\leq w(b)$ for all $b$ such
that $b\sim a$.
\end{corollary}
\begin{proof}
Since $b\sim b^o$ and $w(b)\geq w(b^o)$, we can assume that $b$ is stable. It
is readily that if $a\sim b$, then $a_{\max}-a\sim a_{\max}-b$. Therefore, the
conclusion of the statement is a consequence of Theorems \ref{T:max} and
\ref{T:dua}.
\end{proof}

Concerning to the rank of non-negative configurations on complete graphs, investigations by Cori and Le Borgne \cite{CL13} lead to the computing of the super-stable configuration equivalent to a given configuration, which is also called the super-stabilization of a configuration. A natural way is using the duality between critical and super-stable configuration. Precisely, if we denote $ss(a)$ (resp. $crit(a)$) the super-stable (resp. critical) configuration equivalent to a stable configuration $a$, then $$ss(a)=a_{\max}-crit(a_{\max}-a).$$ The rest is applying Lemma \ref{L:HLM2}(iii) for computing $crit(a_{\max}-a)$. However, we could do it directly and deterministically by using $\sigma^M$ as follows: 
\begin{itemize}
\item[i)] Calculate $a^o$;
\item[ii)] Calculate $a^o-\sigma^M\Delta$;
\item[iii)] We add rows $\Delta_i$ of $\Delta$ satisfying $(a^o-\sigma^M\Delta)_i<0$ to $a^o-\sigma^M\Delta$. 
\item[iv)] Repeat step (iii) until we obtain a non-negative configuration which is the super-stable configuration in the equivalent class of $a$.
\end{itemize}
This could be considered as a deterministic way which is a generalization of the way Cori and Le Borgne have done on complete graph in \cite{CL13}. 

Now we present an example to illustrate the results.

Let $G$ be a graph with $V=\{v_1,v_2,v_3,s\}$ and
$$E=\{(v_1,v_2)^3,(v_1,v_3)^2,(v_2,v_3),(v_2,s),(v_3,v_1),(v_3,v_2)\}.$$ Then $$\Delta=\left(\begin{matrix}
5&-3&-2\\
0&2&-1\\
-1&-1&2
\end{matrix}\right); \quad\quad \sigma^M=(1,3,3);\quad\quad \sigma^M\Delta=(2,0,1).$$
$G$ has $8$ critical configurations and also $8$ super-stable configurations,  
denoted by $Crit(G)$ and $SS(G)$ respectively, as follows:
\begin{align*}
Crit(G)&=\{(4,1,1), (4,1,0), (4,0,1), (3,1,1), (3,1,0), (3,0,1), (2,1,1), (2,0,1)\};\\
SS(G)&=\{(0,0,0), (0,0,1), (0,1,0), (1,0,0), (1,0,1), (1,1,0), (2,0,0), (2,1,0)\};
\end{align*}
To find the super-stable configuration in the equivalent class of $(4,0,0)$ for instance, we do as follows: 
Consider configuration $(4,0,0)-\sigma^M\Delta$, which equals $(2,0,-1)$. Since its third component is negative, we add the $3$th row $(-1,-1,2)$ of $\Delta$ to $(2,0,-1)$ and get $(1,-1,1)$. Since the second component of $(1,-1,1)$ is negative, we add $\Delta_2=(0,2,-1)$ to it and get $(1,1,0)$ of non-negative components. Hence, $(1,1,0)$ is the super-stable configuration equivalent to $(4,0,0)$. Similarly, we could find equivalent classes on the set of stable configurations of $G$ are:
\begin{align*}
(4,1,1)&\sim (2,1,0)\\
(4,1,0)&\sim (1,0,1)\\
(4,0,1)&\sim (2,0,0)\sim (1,1,1)\\
(3,1,1)&\sim (1,1,0)\sim (4,0,0)\\
(3,1,0)&\sim (0,0,1)\\
(3,0,1)&\sim (1,0,0)\sim (0,1,1)\\
(2,1,1)&\sim (0,1,0)\\
(2,0,1)&\sim (0,0,0).
\end{align*}
We also notice that maximal super-stable configurations on digraphs do not have the same weight. For instance, super-stable configurations $(2,1,0)$ and $(1,0,1)$ shown above. 
\bibliographystyle{plain}
\bibliography{biblio}
\end{document}